\documentclass[12pt]{amsart}

\newtheorem{theorem}{Theorem}[section]
\newtheorem{lemma}{Lemma}[section]

\theoremstyle{definition}
\newtheorem{definition}{Definition}[section]
\newtheorem{corollary}{Corollary}[section]
\newtheorem{prop}{Proposition}[section]
\newtheorem{example}{Example}[section]

\theoremstyle{remark}

\numberwithin{equation}{section}

\begin{document}
\title[Characterization of warped product submanifolds of $(LCS)_n$-Manifolds]{Characterization of warped product submanifolds of Lorentzian
concircular structure Manifolds}
\author[S. K. Hui$^*$, L-.I. Piscoran and T. Pal]{SHYAMAL KUMAR HUI$^*$, LAURIAN-IOAN PISCORAN and TANUMOY PAL}
 \keywords{$(LCS)_n$-manifold, CR- submanifold, pseudo slant submanifold, bi-slant submanifold,
 warped product submanifold\\$^*$Corresponding author}.
\begin{abstract}
Recently Hui et al. (\cite{HAP}, \cite{HAN}) studied contact CR-warped product submanifolds and also warped product
pseudo-slant submanifolds of a $(LCS)_n$-manifold $\bar{M}$. In this paper we have studied the characterization for both
these classes of warped product submanifolds. It is also shown that there do not exists any proper warped product bi-slant
submanifold of a $(LCS)_n$-manifold. Although we constructed an example of a bi-slant submanifold of $(LCS)_n$-manifold.
\end{abstract}
\maketitle
\section{Introduction}
 As a generalization of Riemannian product manifold  Bishop and O'Neill \cite{BISHOP} introduced the notion of warped product manifold and
later it was studied in (\cite{ATCEKEN1}, \cite{ATCEKEN2}, \cite{HUOM}, \cite{UD1}, \cite{UDDIN}-\cite{UKK}).
 The existence or non-existence of warped product manifolds plays
some important role in differential geometry as well as physics.\\
\indent As a generalization of LP-Sasakian manifold introduced independently by  Matsumoto \cite{8} and also by Mihai and
Rosca \cite{9}, Shaikh \cite{11} introduced the notion of Lorentzian concircular structure manifolds
(briefly, $(LCS)_n$-manifolds) with an example. Then Shaikh and Baishya (\cite{12}, \cite{13})
investigated the applications of $(LCS)_n$-manifolds to the general
theory of relativity and cosmology. The $(LCS)_n$-manifolds are also
studied in (\cite{SKH}, \cite{SHAIKH1}, \cite{SHAIKH2}-\cite{14}).\\
\indent Due to important applications in applied mathematics and theoretical physics, the geometry of submanifolds
has become a subject of growing interest. Analogous to almost Hermitian manifolds, the invariant and anti-invariant
subamnifolds are depend on the behaviour of almost contact metric structure $\phi$. The study of the differential geometry
of a contact CR-submanifold as a generalization of invariant and anti-invariant subamnifold was introduced by Bejancu \cite{BEJ1}.
In this connection it is mentioned that different class of submanifolds of $(LCS)_n$-manifolds are studied in (\cite{ATCE2}, \cite{HUI2}, \cite{HUI1},
 \cite{HAP}, \cite{HP}, \cite{HPP}, \cite{SHAIKH9}).  Recently Hui et al. (\cite{HAP}, \cite{HAN}) studied contact CR-warped product submanifolds and also warped product pseudo-slant submanifolds of a $(LCS)_n$-manifold $\bar{M}$. In this paper we have studied the characterization for both
these classes of warped product submanifolds. An example of bi-slant submanifold of $(LCS)_n$-manifold is 
constructed. However, it is also shown that there do not exists any proper warped product bi-slant
submanifold of a $(LCS)_n$-manifold.
\section{Preliminaries}
Let $\bar{M}$ be an $n$-dimensional Lorentzian manifold \cite{NIL} admitting a unit
timelike concircular vector field $\xi$, called the characteristic
vector field of the manifold. Then we have
\begin{equation}
\label{2.1}
g(\xi, \xi)=-1.
\end{equation}
Since $\xi$ is a unit concircular vector field, it follows that
there exists a non-zero 1-form $\eta$ such that for
\begin{equation}
\label{2.2}
g(X,\xi)=\eta(X),
\end{equation}
the equation of the following form holds \cite{15}
\begin{equation}
\label{2.3}
(\bar\nabla _{X}\eta)(Y)=\alpha \{g(X,Y)+\eta(X)\eta(Y)\},
\ \ \ (\alpha\neq 0)
\end{equation}
\begin{equation}
\label{2.4}
\bar\nabla _{X}\xi = \alpha \{X +\eta(X)\xi\}, \ \ \ \alpha\neq 0
\end{equation}
for all vector fields $X$, $Y$, where $\bar{\nabla}$ denotes the
operator of covariant differentiation with respect to the Lorentzian
metric $g$ and $\alpha$ is a non-zero scalar function satisfies
\begin{equation}
\label{2.5}
{\bar\nabla}_{X}\alpha = (X\alpha) = d\alpha(X) = \rho\eta(X),
\end{equation}
$\rho$ being a certain scalar function given by $\rho=-(\xi\alpha)$.
If we put
\begin{equation}
\label{2.6}
\phi X=\frac{1}{\alpha}\bar\nabla_{X}\xi,
\end{equation}
then from (\ref{2.4}) and (\ref{2.6}) we have
\begin{equation}
\label{2.7}
 \phi X = X+\eta(X)\xi,
\end{equation}
\begin{equation}
\label{2.8}
g(\phi X,Y) = g(X,\phi Y)
\end{equation}
from which it follows that $\phi$ is a symmetric (1,1) tensor and
called the structure tensor of the manifold. Thus the Lorentzian
manifold $\bar{M}$ together with the unit timelike concircular vector
field $\xi$, its associated 1-form $\eta$ and an (1,1) tensor field
$\phi$ is said to be a Lorentzian concircular structure manifold
(briefly, $(LCS)_{n}$-manifold), \cite{11}. Especially, if we take
$\alpha=1$, then we can obtain the LP-Sasakian structure of
Matsumoto \cite{8}. In a $(LCS)_{n}$-manifold $(n>2)$ $\bar{M}$, the following
relations hold \cite{11}:
\begin{equation}
\label{2.9}
\eta(\xi)=-1,\ \ \phi \xi=0,\ \ \ \eta(\phi X)=0,\ \ \
g(\phi X, \phi Y)= g(X,Y)+\eta(X)\eta(Y),
\end{equation}
\begin{equation}
\label{2.10}
\phi^2 X= X+\eta(X)\xi,
\end{equation}
\begin{equation}
\label{2.11}
(\bar{\nabla}_{X}\phi)Y=\alpha\{g(X,Y)\xi+2\eta(X)\eta(Y)\xi+\eta(Y)X\},
\end{equation}
\begin{equation}
\label{2.12}
(X\rho)=d\rho(X)=\beta\eta(X)
\end{equation}
for all $X,\ Y,\ Z\in\Gamma(T\bar{M})$ and $\beta = -(\xi\rho)$ is a scalar function.\\
Let $M$ be a submanifold of  $\bar{M}$ with
induced metric $g$. Also let $\nabla$ and $\nabla^{\perp}$ are the
induced connections on the tangent bundle $TM$ and the normal bundle
$T^{\perp}M$ of $M$ respectively. Then the Gauss and Weingarten
formulae are given by
\begin{equation}
\label{2.13} \bar{\nabla}_{X}Y = \nabla_{X}Y + h(X,Y)
\end{equation}
and
\begin{equation}
\label{2.14} \bar{\nabla}_{X}V = -A_{V}X + \nabla^{\perp}_{X}V
\end{equation}
for all $X,Y\in\Gamma(TM)$ and $V\in\Gamma(T^{\perp}M)$, where $h$
and $A_V$ are second fundamental form and the shape operator
(corresponding to the normal vector field $V$) respectively for the
immersion of $M$ into $\bar{M}$ and they are related by
$\label{2.15} g(h(X,Y),V) = g(A_{V}X,Y)$,
for any $X,Y\in\Gamma(TM)$ and $V\in\Gamma(T^{\perp}M)$.\\
For any $X\in \Gamma(TM)$ and $V\in \Gamma(T^\bot M)$, we can write
\begin{eqnarray}
\label{2.16}
\text{(a)} \  \phi X = PX+QX,\ \  \text{(b)}\  \phi V= bV+cV
\end{eqnarray}
where $\phi X,\ bV$ are the tangential components and $QX,\ cV$ are the normal components.\\
A submanifold $M$ of a $(LCS)_n$-manifold $\bar{M}$ is said to be invariant if $\phi(T_pM)\subseteq T_pM$, for every $p\in M$ and
anti-invariant if $\phi T_pM\subseteq T^\bot_pM$, for every $p\in M$.\\
A submanifold $M$ of a $(LCS)_n$-manifold $\bar{M}$ is said to be a CR-submnaifold if there is a differential distribution
$\mathcal{D}:p\rightarrow \mathcal{D}_p \subseteq T_p M$ such that $\mathcal{D}$ is an invariant distribution and the orthogonal
complementary distribution $\mathcal{D}^\bot$ is anti-invariant.\\
The normal space of a CR-submanifold $M$ is decomposed as $T^\bot M =Q\mathcal{D^\bot}\oplus\nu$,
where $\nu$ is the invariant normal subbundle of $M$ with respect to $\phi$.\\
A submanifold $M$ of a $(LCS)_n$-manifold $\bar{M}$ is said to be slant if for each non-zero vector $X\in T_pM$ the angle between $\phi X$
and $T_pM$ is a constant, i.e. it does not depend on the choice of $p\in M$.\\
A submanifold $M$ of a $(LCS)_n$-manifold $\bar{M}$ is said to be a pseudo-slant submanifold if there exists a pair of
orthogonal distributions $\mathcal{D^\bot}$ and $\mathcal{D^\theta}$ such that \\
(i) TM admits the orthogonal direct decomposition $TM=\mathcal{D^\bot}\oplus\mathcal{D^\theta}$,\\
(ii) The distribution $\mathcal{D^\bot}$ is anti-invariant, \\
(iii) The distribution $\mathcal{D^\theta}$ is slant with slant angle $\theta\neq 0,\frac{\pi}{2}$.\\
From the definition it is clear that if $\theta=0$, then $M$ is a CR-submanifold. We say that a pseudo-slant submanifold is proper if
$\theta\neq 0,\frac{\pi}{2}$. The normal space of a pseudo-slant submanifold $M$ is decomposed as $T^\bot M =Q\mathcal{D^\theta}\oplus\phi\mathcal{D}^\bot\oplus\nu$.\\
On a slant submanifold $M$ of a $(LCS)_n$-manifold $\bar{M}$, we have \cite{HUI2}
\begin{eqnarray}
\label{2.17}  P^2X &=&\cos^2\theta[X+\eta(X)\xi],  \\
\label{2.18} Q^2X &=&\sin^2\theta[X+\eta(X)\xi],
\end{eqnarray}
where $\theta$ is the slant angle of $M$ in $\bar{M}$.\\
From (\ref{2.17}) and (\ref{2.18}), we get
\begin{eqnarray}
\label{2.19}g(PX,PY) &=& \cos^2\theta[g(X,Y)+\eta(X)\eta(Y)] ,\\
\label{2.20}g(QX,QY) &=& \sin^2\theta[g(X,Y)+\eta(X)\eta(Y)],
\end{eqnarray}
for any $X,\ Y\in\Gamma(TM)$.\\
Also for a slant submanifold from (\ref{2.16}) and (\ref{2.17}), we have
\begin{equation}\label{2.21}
bQX=\sin^2\theta(X+\eta(X)\xi)\ \ \text{and}\ \ cQX=-QPX
\end{equation}
\indent For a Riemannian manifold $\bar{M}$ of dimension $n$ and a smooth function $f$ on
$\bar{M}$, $\nabla f$, the gradient of $f$ which is defined by
\begin{equation}
\label{2.19}
g(\nabla f,X)=X(f)(\text{or}=(Xf))
\end{equation}
for any $X\in \Gamma (TM)$.\\

\begin{definition}\cite{BISHOP}
Let $(N_1,g_1)$ and $(N_2,g_2)$ be two Riemannian manifolds with Riemannian metric $g_1$
and $g_2$ respectively and $f$ be a positive definite smooth function on $N_1$. The warped product
of $N_1$ and $N_2$ is the Riemannian manifold $N_1\times_{f}N_2 = (N_1\times N_2,g)$, where
\begin{equation}
\label{2.20}
g=g_1+f^2g_2.
\end{equation}
\end{definition}
\noindent A warped product manifold $N_1\times_{f}N_2$ is said to be trivial if the warping function $f$ is constant.\\
\begin{prop}\cite{NIL}
Let $M=N_1\times_{f}N_2$ be a warped product manifold. Then
\begin{eqnarray*}
\nabla_UX = \nabla_XU = (X\ln f) U,
\end{eqnarray*}
for any $X$, $Y\in\Gamma(TN_1)$ and $U\in\Gamma(TN_2)$.
\end{prop}
\begin{theorem}[Hiepko's Theorem]\cite{HIPKO}
Let $\mathcal{D}_1$ and $\mathcal{D}_2$ be two orthogonal complementary distributions on a Riemannian manifold $M$. Suppose that $\mathcal{D}_1$ and $\mathcal{D}_2$ are both involutive such that $\mathcal{D}_1$ is a totally geodesic foliation and $\mathcal{D}_2$ is a spherical foliation. Then $M$ is locally isometric to a non trivial warped product $M_1\times_f M_2$, where $M_1$ and $M_2$ are integral manifolds of $\mathcal{D}_1$ and $\mathcal{D}_2$,
respectively.
\end{theorem}
\section{characterization for contact CR-warped product submanifolds}
In \cite{HAN} it is shown that contact CR-warped product submanifolds of $\bar{M}$ of the form
$N_\bot\times_fN_T$, where $N_T$ and $N_\bot$ are invariant and anti-invariant submanifolds of $\bar{M}$ respectively,
 exists if $\xi \in \Gamma (TN_\bot)$ and does not exists if $\xi \in \Gamma (TN_T)$.
In this section we find a characterization for a submanifolds $M$ of $\bar{M}$ to be contact CR-warped product of the form
 $N_\bot\times_fN_T$ such that $\xi \in \Gamma(TN_\bot)$. First we prove the following Lemma:
\begin{lemma}
Let $M=N_\bot\times_fN_T$ be a warped product submanifold of $\bar{M}$ such that
$\xi \in \Gamma(T N_\bot)$, then
\begin{eqnarray}
\label{3.1}
 g(h(X,Y),\phi Z)=-\alpha \eta(Z)g(X,Y)-(Z\ln f)g(\phi X,Y)
\end{eqnarray}
for $X,\ Y\in \Gamma(TN_T)$ and $Z\in \Gamma(TN_\bot)$.
\end{lemma}
\begin{proof}
 For $X,\ Y\in \Gamma(TN_T)$ and $Z,\xi\in \Gamma(TN_\bot)$ we have
from (\ref{2.11}) and (\ref{2.13}) that
\begin{eqnarray*}
  g(h(X,Y),\phi Z) &=& g(\bar{\nabla}_XY,\phi Z )\\
   &=&g(\bar{\nabla}_X\phi Y,Z)-g((\bar{\nabla}_X\phi)Y,Z)  \\
   &=&-g(\phi Y,\nabla_XZ)-\alpha g(X,Y)\eta(Z).
\end{eqnarray*}
By virtue of Proposition 2.1 from the above relation we get (\ref{3.1}).
\end{proof}
\par Now interchanging $X$ by $\phi X$ and $Y$ by $\phi Y$, we get the following respective relations
\begin{eqnarray}
\label{3.2}  g(h(\phi X,Y),\phi Z) &=& -\alpha \eta(Z)g(\phi X,Y)-Z(\ln f)g(X,Y), \\
\label{3.3}  g(h(X,\phi Y),\phi Z) &=& -\alpha \eta(Z)g(\phi X,Y)-Z(\ln f)g(X,Y), \\
\label{3.4}  g(h(\phi X,\phi Y),\phi Z) &=& -\alpha \eta(Z)g(X,Y)-Z(\ln f)g(\phi X,Y).
\end{eqnarray}
\begin{corollary}
Let $M=N_\bot\times_fN_T$ be a warped product submanifold of $\bar{M}$ such that
$\xi \in \Gamma(T N_\bot)$. Then
\begin{eqnarray*}
 g(h(\phi X,Y),\phi Z)   &=& g(h(X,\phi Y),\phi Z) \\
\text{and}\ \ \ \ \  g(h(\phi X,\phi Y),\phi Z)&=&g(h(X,Y),\phi Z)
\end{eqnarray*}
for $X,\ Y\in \Gamma(TN_T)$ and $Z\in \Gamma(TN_\bot)$.
\end{corollary}
Now we have the following characterization theorem:
\begin{theorem}
Let $M$ be a contact CR-submanifold of a $(LCS)_n$-manifold $\bar{M}$ such that $\xi$ is
tangent to the anti-invariant distribution $\mathcal{D}^\bot$. Then $M$ is locally a warped product submanifold if and only if
 \begin{equation}
 \label{3.5}
 A_{\phi Z}X=-\alpha \eta(Z)X-(Z\mu)\phi X
 \end{equation}
 for any $X\in\Gamma(\mathcal{D})$ and $Z\in \Gamma(\mathcal{D}^\bot)$ and also for some smooth function $\mu$ on $M$ such that
 $(Y\mu)=0$ for any $Y\in D$.
\end{theorem}
\begin{proof}
\indent If $M$ be a contact CR-warped product submanifold, then for any $X\in \Gamma(TM_T)$ and $Z,\ W\in \Gamma (TM_\bot)$,
 we have
 \begin{equation*}
   g(A_{\phi Z}X,W)-g(h(X,W),\phi Z)=g(\bar{\nabla}_WX,\phi Z)=g(\phi\bar{\nabla}_WX,Z).
 \end{equation*}
 Using (\ref{2.11}) in the above equation we get
 \begin{equation}\label{3.5a}
   g(A_{\phi Z}X,W)=g(\bar{\nabla}_W\phi X,Z).
 \end{equation}
 Then from (\ref{2.13}) and Proposition 2.1 we have from (\ref{3.5a}) that $g(A_{\phi Z}X,W)=0$ and therefore $A_{\phi Z}X$ has no
 component in $\Gamma(TN_\bot)$. Hence by virtue of Lemma 3.1, the relation (\ref{3.5}) follows.\\
  \par Conversely, let $M$ be a contact CR-submanifold of $\bar{M}$ with the invariant and anti-invariant distributions
  $D$ and $\mathcal{D}^\bot$ such that the relation (\ref{3.5}) holds. Then for any $X\in \Gamma(\mathcal{D})$ and $Z,\ W\in \Gamma(\mathcal{D}^\bot)$, and
   using (\ref{2.13}) we have
  \begin{eqnarray*}
  g(\nabla_ZW,\phi X)&=& g(\bar{\nabla}_Z\phi W,X)-g((\bar{\nabla}_Z\phi)W,X)\\
  &=&-g(\phi W,\bar{\nabla}_ZX)\\
  &=&-g(A_{\phi W}X,Z).
  \end{eqnarray*}
  Using (\ref{3.5}) in above relation we get $$g(\nabla_ZW,\phi X)=0.$$\\
  Similarly, we get $g(\nabla_WZ,\phi X)=0.$
  Thus we obtain
  \begin{equation}
  \label{3.6}
 g(\nabla_ZW+\nabla_WZ,\phi X)=0,
  \end{equation}
  Which implies that $\nabla_ZW+\nabla_WZ\in \Gamma(\mathcal{D}^\bot)$, i.e., $\mathcal{D}^\bot$ is integrable and its leaves are totally geodesic in $M$.
  Again for any $X,\ Y\in \Gamma(\mathcal{D})$ and $Z\in \Gamma(\mathcal{D}^\bot)$, we get
  \begin{eqnarray}
  \label{3.7}
  g(\nabla_XY,Z)&=&g(\bar{\nabla}_X\phi Y,\phi Z)+\eta(Z)g(Y,\bar{\nabla}_X\xi).
  \end{eqnarray}
  Using (\ref{2.6}) in (\ref{3.7}), we get
  \begin{equation}
  \label{3.8}
 g(\nabla_XY,Z)=g(h(X,\phi Y),\phi Z)+\alpha \eta(Z)g(Y,\phi X).\\
  \end{equation}
  Interchanging $X$ and $Y$ in (\ref{3.8}), we get
 \begin{equation}\label{3.8a}
 g(\nabla_YX,Z)=g(h(\phi X,Y),\phi Z)+\alpha \eta(Z)g(X,\phi Y).
 \end{equation}
From (\ref{3.8}) and (\ref{3.8a}), we have
  \begin{equation}
  \label{3.9}
 g([X,Y],Z)=g(h(X,\phi Y),\phi Z)-g(h(\phi X,Y),\phi Z).
  \end{equation}
  Using (\ref{3.5}) in (\ref{3.9}), we get $g([X,Y],Z)=0$ and therefore $\mathcal{D}$ is integrable on $M$.\\
  Let us consider a leaf $N_T$ of $\mathcal{D}$ in $M$ and let $h^T$ be the second fundamental form
  of $N_T$ in $M$, then we have
  \begin{eqnarray}
  \label{3.10}
    g(h^T(X,Y),Z) &=& g(\phi\bar{\nabla}_YX,\phi Z)-\eta(Z)g(\bar{\nabla}_YX,\xi)\\
\nonumber&=&-g(\phi X,\bar{\nabla}_Y\phi Z)+\eta(Z)g(X,\bar{\nabla}_Y\xi).
  \end{eqnarray}
  Using (\ref{2.6}) and (\ref{2.14}), (\ref{3.10}) yields
  \begin{equation}\label{3.12}
  g(h^T(X,Y),Z)=g(\phi X,A_{\phi Z}Y)+\alpha \eta(Z)g(X,\phi Y).
  \end{equation}
  From (\ref{3.5}) and (\ref{3.12}), we obtain
  \begin{equation}\label{3.13}
  g(h^T(X,Y),Z)=-(Z\mu)g(X,Y).
  \end{equation}
  Using (\ref{2.19}) in (\ref{3.13}), we get
  \begin{eqnarray}
    h^T(X,Y) &=& -(\nabla\mu)g(X,Y),
  \end{eqnarray}
  where $\nabla\mu$ is the gradient of the function $\mu$ and therefore $N_T$ is totally umbilical in $M$ with
  mean curvature $-\nabla\mu$. Moreover, the condition $(Y\mu)=0$, for any $Y\in \mathcal{D}$ implies that the leaves
  of $\mathcal{D}$ are extrinsic spheres in $M$, i.e., the integral manifold $N_T$ of $\mathcal{D}$ is umbilical and its mean curvature
  vector field is non zero and parallel along $N_T$. Hence by Hiepko's theorem $M$ is locally a warped product $N_\bot\times_fN_T$ ,
  where $N_T$ and $N_\bot$ denote the integral manifolds of the distributions $\mathcal{D}$ and $\mathcal{D}^\bot$ respectively and $f$ is the warping
  function. Thus the theorem is proved completely.
\end{proof}
\section{characterization for warped product pseudo slant submanifolds}
Recently Hui et al. \cite{HAP} studied warped product pseudo-slant submanifolds  of $(LCS)_n$-manifolds. In this section we
 obtain a characterization for a submanifold $M$ of $\bar{M}$ to be a warped product pseudo-slant submanifold of the form
 $N_\theta\times_fN_\bot$, where $N_\theta$ is a slant submanifold tangent to $\xi$ and
 $N_\bot$ is an anti-invariant submanifolds of $\bar{M}$.
\begin{lemma}
Let $M$ be a proper pseudo-slant submanifold of a $(LCS)_n$-manifold $\bar{M}$ with anti-invariant and proper slant dfistributions
$\mathcal{D}^\bot$ and $\mathcal{D}^\theta$, respectively such that $\xi\in\Gamma(\mathcal{D}^\theta)$. Then
\begin{equation}\label{4.1}
g(\nabla_XY,Z)=\sec^2\theta[g(h(X,PY),\phi Z)+g(h(X,Z),QPY)],
\end{equation}
for any $X,\ Y\in \Gamma(\mathcal{D}^\theta)$ and $Z\in \Gamma(D^\bot)$.
\end{lemma}
\begin{proof}
For any $X,\ Y\in\Gamma(\mathcal{D}^\theta)$ and $Z\in\Gamma(\mathcal{D}^\bot)$, we have
\begin{eqnarray*}
  g(\nabla_XY,Z) &=& g(\phi \bar{\nabla}_XY,\phi Z) \\
   &=&g(\bar{\nabla}_X\phi Y,\phi Z)-g((\bar{\nabla}_X\phi)Y,Z) \\
   &=&g(\bar{\nabla}XPY,\phi Z)+g(\bar{\nabla}_XQY,\phi Z) \\
   &=&g(h(X,PY),\phi Z)+g(\bar{\nabla}X\phi QY,Z)-g((\bar{\nabla}_X\phi)QY,\phi Z) \\
   &=&g(h(X,PY),\phi Z)+g(\bar{\nabla}_XbQY,Z)+g(\bar{\nabla}_XcQY,Z).
\end{eqnarray*}
Using (\ref{2.21}) in the above relation, we get
\begin{eqnarray*}
  g(\nabla_XY,Z) &=&g(h(X,PY),\phi Z)+\sin^2\theta g(\bar{\nabla}_XY,Z)-g(\bar{\nabla}_XQPY,Z).
\end{eqnarray*}
From which (\ref{4.1}) follows.
\end{proof}
\begin{corollary}
 Let $M$ be a proper pseudo-slant submanifold of a $(LCS)_n$-manifold $\bar{M}$ with anti-invariant and proper slant distributions
$\mathcal{D}^\bot$ and $\mathcal{D}^\theta$, respectively such that $\xi\in\Gamma(\mathcal{D}^\theta)$. Then the distribution $\mathcal{D}^\theta$
defines a totally geodesic foliation if and only if
\begin{equation*}
g(h(X,PY),\phi Z)+g(h(X,Z),QPY)=0
\end{equation*}
for every $X,\ Y\in \Gamma(\mathcal{D}^\theta)$ and $Z\in \Gamma(\mathcal{D}^\bot)$.
\end{corollary}
\begin{lemma}
Let $M=N_\theta\times_fN_\bot$ be a warped product submanifold of a $(LCS)_n$-manifold $\bar{M}$, where $N_\bot$ and $N_\theta$ are ant-invariant
and proper slant submanifold of $\bar{M}$ such that $\xi\in \Gamma(TN_\theta)$. Then
\begin{eqnarray}
\label{4.2} g(h(X,Y),\phi Z)+g(h(X,Z),QY) &=&0, \\
\label{4.3} g(h(Z,W),QX)+g(h(X,Z),QW) &=& (\phi X\ln f)g(Z,W), \\
\label{4.4}\ \ \  g(h(Z,W),QPX)+g(h(PX,Z),QW) &=& \cos^2\theta[(X\ln f)+\alpha \eta(X)]g(Z,W).
\end{eqnarray}
\end{lemma}
\begin{proof}
For any $X,\ Y\in \Gamma(TN_\theta)$ and $Z\in \Gamma(TN_\bot)$, we have
\begin{eqnarray*}
g(h(X,Y),\phi Z) &=& g(\bar{\nabla}_XY,\phi Z) \\
   &=&g(\phi \bar{\nabla}_XY,Z) \\
   &=&g(\bar{\nabla}_X\phi Y,Z) \\
   &=&g(\bar{\nabla}_XPY,Z)+g(\bar{\nabla}_XQY,Z).
\end{eqnarray*}
Using Proposition 2.1 in the above relation, we get
\begin{equation}\label{4.5}
g(h(X,Y),\phi Z)=(X\ln f)g(Z,PY)-g(h(X,Z),QY).
\end{equation}
Thus (\ref{4.2}) follows from (\ref{4.5}).
Also, for any $X\in\Gamma(TN_\theta)$ and $Z,\ W\in \Gamma(TN_\bot)$, we have
\begin{eqnarray*}
g(h(Z,W),QX)&=&g(\bar{\nabla}_ZW,\phi X)-g(\bar{\nabla}_ZW,PX) \\
   &=&g(\bar{\nabla}_Z\phi W,X)-g((\bar{\nabla}_Z\phi)W,X)-g(\bar{\nabla}_ZW,PX) \\
   &=&-g(h(X,Z),\phi W)+g(W,\bar{\nabla}_ZPX).
\end{eqnarray*}
Using Proposition 2.1 in the above relation we get (\ref{4.3}).
Interchanging $X$ by $PX$ in (\ref{4.3}) we get (\ref{4.4})
\end{proof}
Now, we prove the following characterization theorem for warped product pseudo-slant submanifolds.
\begin{theorem}
Let $M$ be a proper pseudo-slant submanifold of a $(LCS)_n$-manifold $\bar{M}$ with anti-invariant distribution $\mathcal{D}^\bot$
and proper pseudo-slant distribution $\mathcal{D^\theta}$, respectively such that $\xi\in \Gamma(\mathcal{D}^\theta)$. Then $M$ is locally
a mixed-geodesic warped product submanifold of the form $N_\theta\times_fN_\bot$ if and only if
\begin{equation}\label{4.6}
A_{\phi Z}X=0\ \ \text{and}\ \  A_{QPX}Z=\cos^2\theta[(X\mu)+\alpha \eta(X)]Z,
\end{equation}
for any $X\in \Gamma(\mathcal{D^\theta}),\ Z\in \Gamma(\mathcal{D^\bot})$ and for some function $\mu$ on $M$ satisfying $(Z\mu)=0$, for
any $Z\in \Gamma(\mathcal{D}^\bot)$.
\end{theorem}
\begin{proof}
  Let $M=N_\theta\times_fN_\bot$ be a mixed geodesic warped product submanifold of a $(LCS)_n$-manifold $\bar{M}$ such that (\ref{4.6}) holds.
  Then for any $X,\ Y\in \Gamma(\mathcal{D}^\theta)$ and $Z\in\Gamma(\mathcal{D}^\bot)$, from (\ref{4.2}) and (\ref{4.4}), we get (\ref{4.6}).\\
  Conversely, Let $M$ is a proper pseudo-slant submanifold of a $(LCS)_n$-manifold $\bar{M}$ such that (\ref{4.6}) holds.\\
  Then for any for any $X,\ Y\in \Gamma(\mathcal{D}^\theta)$ and for any $Z\in \Gamma(\mathcal{D}^\bot)$, from (\ref{4.1}) and (\ref{4.6}), we get
  $g(\nabla_XY,Z)=0$ and hence the leaves of $\mathcal{D}^\theta$ are totally geodesic in $M$.\\ Also, for any $X\in\Gamma(\mathcal{D})$ and
  $Z,\ W\in\Gamma(\mathcal{D}^\bot)$, we have
  \begin{eqnarray*}
  g([Z,W],X) &=& g(\bar{\nabla}_ZW,X)-g(\bar{\nabla}_WZ,X) \\
      &=& g(\phi\bar{\nabla}_ZW,\phi X)-g(\phi\bar{\nabla}_WZ,\phi X) \\
      &=& g(\bar{\nabla}_Z\phi W,\phi X)-g(\bar{\nabla}W\phi Z,\phi X) \\
      &=& -g(\phi W,\bar{\nabla}_Z\phi X)+g(\phi Z,\bar{\nabla}_W\phi X) \\
      &=& -g(\phi W,\bar{\nabla}_ZPX)-g(\phi W,\bar{\nabla}_ZQX)+g(\phi Z,\bar{\nabla}_WPX)+g(\phi Z,\bar{\nabla}_WQX) \\
      &=& -g(\phi W,h(Z,PX))-g(W,\bar{\nabla}_ZbQX)-g(W,\bar{\nabla}_ZcQX)\\
      &&+g(\phi Z,h(W,PX))+g(Z,\bar{\nabla}_WbQX)+g(Z,\bar{\nabla}_WcQX).
  \end{eqnarray*}
  Using (\ref{2.21}) in the above relation, we get
  \begin{eqnarray}
  \label{4.7}g([Z,W],X) &=& -g(A_{\phi W}PX,Z)+g(A_{\phi Z}PX,W)+\sin^2\theta g([Z,W],X) \\
  \nonumber&&+g(A_{QPX}Z,W)-g(A_{QPX}W,Z).
  \end{eqnarray}
  Using (\ref{4.6}) in (\ref{4.7}), we get
  \begin{equation}\label{4.8}
  \cos^2\theta g([Z,W],X)=0.
  \end{equation}
  Since $\mathcal{D}^\theta$ is proper pseudo-slant so, $\theta\neq 0,\frac{\pi}{2}$. Therefore, $g([Z,W],X)=0$ and hence the anti-invariant distribution
  $\mathcal{D}^\bot$ is integrable.\\
  Now, let $h^\bot$ be the second fundamental form of a leaf $N_\bot$ of $\mathcal{D}^\bot$ in $M$. Then for any $Z,\ W\in \Gamma(\mathcal{D}^\bot)$ and $X\in \Gamma(\mathcal{D}^\theta)$, we have
  \begin{eqnarray*}
    g(h^\bot(Z,W),X) &=& g(\phi \bar{\nabla}_ZW,\phi X) \\
     &=& g(\bar{\nabla}_Z\phi W,\phi X) \\
     &=& g(\bar{\nabla}_Z\phi W,PX)+g(\bar{\nabla}_Z\phi W,QX) \\
     &=& -g(A_{\phi W}Z,PX)-g(W,\bar{\nabla}_Z\phi QX) \\
     &=& -g(A_{\phi W}PX,Z)-g(W,\bar{\nabla}_ZbQX)-g(W,\bar{\nabla}_ZcQX) \\
     &=& -g(A_{\phi W}PX,Z)+\sin^2\theta g(\bar{\nabla}_ZW,X)-g(A_{QPX}W,Z).
  \end{eqnarray*}
  Therefore
  \begin{equation}\label{4.9}
  \cos^2\theta g(h^\bot(Z,W),X)=-g(A_{\phi W}PX,Z)-g(A_{QPX}W,Z).
  \end{equation}
  Using (\ref{4.6}) in (\ref{4.9}), we get
  \begin{equation}\label{4.10}
  \cos^2\theta g(h^\bot(Z,W),X)=-\cos^2\theta[(X\mu)+\alpha \eta(X)]g(Z,W).
  \end{equation}
  Thus, we get
  \begin{equation*}
  h^\bot(Z,W)=-[\overrightarrow{\nabla}^\bot\mu+\alpha\xi]g(Z,W),
  \end{equation*}
  where $\overrightarrow{\nabla}^\bot\mu$ is gradient of the function $\mu$.\\ Therefore $N_\bot$ is totally umbilical in $M$ with
  the mean curvature $H^\bot=-(\overrightarrow{\nabla}^\bot\mu+\alpha\xi)$.\\ Now, let $\mathcal{D}^N$ be the normal connection of $N_\bot$ in $M$. Then
  for any $Y\in \Gamma(\mathcal{D}^\theta)$ and $Z\in \Gamma(\mathcal{D}^\bot)$, we have
  \begin{equation*}
  g(\mathcal{D}_Z^N\overrightarrow{\nabla}^\bot\mu+\alpha\xi,Y)=g(\nabla_Z\overrightarrow{\nabla}^\bot\mu,Y)+\alpha g(\nabla_Z\xi,Y).
  \end{equation*}
  Also, from (\ref{2.6}) and (\ref{2.13}) we get $\nabla_Z\xi=0$.\\
  Therefore, $g(\mathcal{D}_Z^N\overrightarrow{\nabla}^\bot\mu+\alpha\xi,Y)=g(\nabla_Z\overrightarrow{\nabla}^\bot\mu,Y)=0$, since $(Z\mu)=0$ for
  every $Z\in \Gamma(\mathcal{D}^\bot)$ and hence the mean curvature of $N_\bot$ is parallel.\\ Thus the leaves of the distribution $\mathcal{D}^\bot$
  are totally umbilical in $M$ with non-vanishing parallel mean curvature vector $H^\bot$, i.e. $N_\bot$ is an extrinsic sphere in $M$. Therefore by
  Theorem 2.1, $M$ is a warped product submanifold.
\end{proof}
\section{warped product bi-slant submanifolds of $(LCS)_n$-manifolds}
\begin{definition}
  A submanifold $M$ of a $(LCS)_n$-manifold $\bar{M}$ is said to be a bi-slant submanifold if there exists a pair of orthogonal distributions
  $\mathcal{D}_1$ and $\mathcal{D}_2$ of $M$ such that\\
   $(\textit{i}) TM = \mathcal{D}_1\oplus\mathcal{D}_2 $\\
   $(\textit{ii})\phi \mathcal{D}_1\bot \mathcal{D}_2 \ \ \text{and} \ \   \phi\mathcal{D}_2\bot \mathcal{D}_1$\\
   $(\textit{iii})\mathcal{D}_1, \mathcal{D}_2 $ are slant submanifolds with slant angles $\theta_1$ and $\theta_2$, respectively.
\end{definition}
If we assume $\theta_1=0$ and $\theta_2=\frac{\pi}{2}$, then $M$ is a CR-submanifold and if $\theta_1=0$ and $\theta_2\neq0,\frac{\pi}{2}$, then
$M$ is a semi-slant submanifold. Also, if $\theta_1=\frac{\pi}{2}$ and $\theta_2\neq0,\frac{\pi}{2}$, then $M$ is a pseudo-slant submanifold.
A bi-slant submanifold $M$ of a $(LCS)_n$-manifold $\bar{M}$ is said to be proper if the slant distributions $\mathcal{D}_1$ and
$\mathcal{D}_2$ are of slant angles $\theta_1,\ \theta_2\neq0,\frac{\pi}{2}$.\\
For a proper bi-slant submanifold $M$ of a $(LCS)_n$-manifold, the normal bundle of $M$ is decomposed as $$T^\bot M=Q\mathcal{D}_1\oplus Q\mathcal{D}_2\oplus\nu,$$
where $\nu$ is the invariant normal subbundle of $M$.\\
Now we will construct a bi-slant submanifold of a $(LCS)_n$-manifold.
\begin{example}
 \rm{Consider the semi-Euclidean space ${\mathbb{R}}^{7}$ with the cartesian coordinates $(x_1,y_1\cdots,x_3, y_3,\,t)$ and paracontact structure
\begin{equation*}
\phi\left(\frac{\partial}{\partial x_i}\right)=\frac{\partial}{\partial y_i},\,\,\,\,
\phi\left(\frac{\partial}{\partial y_j}\right)=\frac{\partial}{\partial x_j},\,\,\phi\left(\frac{\partial}{\partial t}\right)=0,\,\,\,\,1\leq i, j\leq3.
\end{equation*}
It is clear that ${\mathbb{R}}^{7}$ is a Lorentzian metric manifold manifold with usual semi-Euclidean metric tensor.
For any $\theta_1, \theta_2\in[0,\frac{\pi}{2}]$ let $M$ be a submanifold of ${\mathbb{R}}^{7}$ defined by
\begin{equation*}
\chi(u, v, w, s,\,t)=(w+u\cos \theta_1,\,u\sin \theta_1,\,s+v\cos \theta_2,\,v\sin \theta_2,\,0,\,0,\,t).
\end{equation*}
 Then the tangent space of $M$ is spanned by the following vectors
\begin{eqnarray*}
  Z_1 &=& \cos \theta_1\frac{\partial}{\partial x_1}+\sin \theta_1\frac{\partial}{\partial x_2}, \\
  Z_2 &=& \cos \theta_2\frac{\partial}{\partial y_1}+\sin \theta_2\frac{\partial}{\partial y_2}, \\
  Z_3 &=& \frac{\partial}{\partial x_1}, \\
  Z_4 &=& \frac{\partial}{\partial y_1},\\
  Z_5 &=&\frac{\partial}{\partial t}.
\end{eqnarray*}
Then we have
\begin{eqnarray*}
  \phi Z_1 &=& \cos \theta_1\frac{\partial}{\partial y_1}+\sin \theta_1\frac{\partial}{\partial y_2}, \\
  \phi Z_2 &=& \cos \theta_2\frac{\partial}{\partial x_1}+\sin \theta_2\frac{\partial}{\partial x_2}, \\
  \phi Z_3 &=& \frac{\partial}{\partial y_1}, \\
  \phi Z_4 &=&\frac{\partial}{\partial x_1},\\
  \phi Z_5 &=& 0.
\end{eqnarray*}
We take ${\mathcal{D}_1}={\rm{Span}}\{Z_1,\,Z_4\}$ and ${\mathcal{D}_2}={\rm{Span}}\{Z_2,\ Z_3\}$, then $g(Z_1,\phi Z_4)=\cos \theta_1$
 and $g(Z_2,\phi Z_3)=\cos\theta_2$. Thus the distributions $\mathcal{D}_1$ and $\mathcal{D}_2$ are slant with slant angles $\theta_1$ and
 $\theta_2$ respectively and hence $M$ is a bi-slant submanifold.
}
\end{example}
\begin{lemma}
Let $M$ be a proper bi-slant submanifold of a $(LCS)_n$-manifold $\bar{M}$ with the slant distributions $\mathcal{D}_1$ and $\mathcal{D}_2$
 such that $\xi\in \Gamma(\mathcal{D}_1)$. Then
 \begin{eqnarray}
 \label{5.1}\cos^2\theta_2g(\nabla_{X_1}X_2,Y_2) &=& g(\nabla_{X_1}PX_2,PY_2)+g(h(X_1,PX_2),QY_2)\\
\nonumber&& +g(h(X_1,Y_2),QPX_2),
 \end{eqnarray}
 for any $X_1\in\Gamma(\mathcal{D}_1)$ and $X_2,\ Y_2\in \Gamma(\mathcal{D}_2)$, where $\theta_1$ and $\theta_2$ are the slant angles
 of slant distributions $\mathcal{D}_1$ and $\mathcal{D}_2$ respectively.
\end{lemma}
\begin{proof}
For any $X_1\in\Gamma(\mathcal{D}_1)$ and $X_2,\ Y_2\in \Gamma(\mathcal{D}_2)$, we have
\begin{eqnarray*}
  g(\nabla_{X_1}X_2,Y_2) &=& g(\phi \bar{\nabla}_{X_1}X_2,\phi Y_2) \\
   &=& g(\bar{\nabla}_{X_1}\phi X_2,\phi Y_2) \\
   &=& g(\bar{\nabla}_{X_1}PX_2,PY_2)+g(\bar{\nabla}_{X_1}PX_2,QY_2)+g(\bar{\nabla}-{X_1}QX_2,\phi Y_2). \\
   &=& g(\nabla_{X_1}PX_2,PY_2)+g(h(X_1,PX_2),QY_2)\\
   &&+g(\bar{\nabla}_{X_1}bQX_2,Y_2)+g(\bar{\nabla}_{X_1}cQX_2,Y_2).
\end{eqnarray*}
Using (\ref{2.21}) in the above relation we get (\ref{5.1}).
\end{proof}
\begin{theorem}
There does not exists a proper warped product bi-slant submanifold $M =M_{1}\times_fM_2$ of $\bar{M}$ such that $\xi\in\Gamma(TM_1)$.
\end{theorem}
\begin{proof}
Let $M=M_{1}\times_fM_2$ be a proper warped product bi-slant submanifold of $\bar{M}$. Then for $X_1\in\Gamma(TM_1)$ and
$X_2,\ Y_2\in \Gamma(TM_2)$, we have
\begin{eqnarray*}
  g(h(X_1,PX_2),QY_2) &=& g(\bar{\nabla}_{X_1}PX_2,\phi Y_2)+g(PX_2\bar{\nabla}_{X_1}PY_2) \\
   &=& \cos^2\theta_2g(\bar{\nabla}_{X_1}X_2,Y_2)-g(h(X_1,Y_2),QPX_2)+g(PX_2,\bar{\nabla}_{X_1}PY_2).
\end{eqnarray*}
Using Proposition 2.1 in the above relation, we get
\begin{equation}\label{5.2}
g(h(X_1,PX_2),QY_2)+g(h(X_1,Y_2),QPX_2)=2\cos^2\theta_2(X_1 \ln f)g(X_2,Y_2).
\end{equation}
Again using Proposition 2.1 in (\ref{5.1}), we get
\begin{equation}\label{5.3}
g(h(X_1,PX_2),QY_2)+g(h(X_1,Y_2),QPX_2)=0.
\end{equation}
From (\ref{5.2}) and (\ref{5.3}), we get
\begin{equation}\label{4.4}
\cos^2\theta_2(X_1\ln f)=0.
\end{equation}
Since $M$ is a proper warped product bi-slant submanifold so, $\theta_2\neq\frac{\pi}{2}$. Therefore $(X_1\ln f)=0$ for every
$X_1\in\Gamma(TM_1)$ and hence $M$ does not exists.
\end{proof}

\vspace{0.1in}
\noindent S. K. Hui and T.PaL\\
Department of Mathematics, The University of Burdwan, Golapbag, Burdwan -- 713104, West Bengal, India\\
E-mail: skhui@math.buruniv.ac.in, tanumoypalmath@gmail.com\\

\noindent L.-I. Piscoran\\
 North University Center of Baia Mare, Technical University of Cluj Napoca, Department of Mathematics and Computer Science, Victoriei 76, 430122 Baia Mare, Romania\\
E-mail: plaurian@yahoo.com\\
\end{document}